\documentclass[a4paper,12pt]{article}
    \usepackage[top=2cm,bottom=2cm,left=2.5cm,right=2.5cm]{geometry}
    \usepackage{cite, amsmath, amssymb}
    \usepackage[margin=0.5cm,%
                font=small,%
                format=hang,%
                labelsep=period,%
                labelfont=bf]{caption}
\usepackage[latin1]{inputenc}
\usepackage{amsmath}
\usepackage{amsfonts}
\usepackage{amssymb}
\usepackage{cite}
\usepackage{amsthm}
\usepackage{makeidx}
\usepackage{graphicx}
\usepackage{mathrsfs,xcolor}
\usepackage{enumerate}
\usepackage{blkarray}
\usepackage{bm}
\usepackage{setspace}
\usepackage{float}
\usepackage{authblk}
\usepackage{multirow}
\usepackage{booktabs}
\usepackage[ruled,vlined]{algorithm2e}
\usepackage{blkarray}
\usepackage{authblk}
\usepackage{multirow}
\usepackage{lineno}

\usepackage{fancyhdr}
\usepackage{lastpage}

\numberwithin{table}{section}
\numberwithin{equation}{section}

\theoremstyle{plain}
\newtheorem{theorem}{Theorem}[section]
\newtheorem{proposition}[theorem]{Proposition}
\newtheorem{definition}[theorem]{Definition}
\newtheorem{lemma}[theorem]{Lemma}
\newtheorem{example}[theorem]{Example}

\newtheorem{corollary}[theorem]{Corollary}
\newtheorem{remark}[theorem]{Remark}

\usepackage{authblk}
\author[1,2]{ \textbf{Bryan S. Hernandez}}
\author[3,4,5,6]{\textbf{Eduardo R. Mendoza}}

\affil[1]{\small \textit{Biomedical Mathematics Group, Pioneer Research Center for Mathematical and Computational Sciences, Institute for Basic Science, Daejeon 34126, Republic of Korea}}
\affil[2]{\small \textit{Institute of Mathematics, University of the Philippines Diliman, Quezon City 1101, Philippines}}
\affil[3]{\small \textit{Mathematics and Statistics Department, De La Salle University, Manila  0922, Philippines}}
\affil[4]{\small \textit{Center for Natural Sciences and Environmental Research, De La Salle University, Manila 0922, Philippines}}
\affil[5]{\small \textit{Max Planck Institute of Biochemistry, Martinsried, Munich, Germany}}
\affil[6]{\small \textit{LMU Faculty of Physics, Geschwister -Scholl- Platz 1, 80539 Munich, Germany}}
\affil[*]{Email addresses: \texttt{bryan.hernandez@upd.edu.ph} and \texttt{eduardo.mendoza@dlsu.edu.ph}}

\title{\textbf{Positive Equilibria of Power Law Kinetics on Networks with Independent Linkage Classes}}
\date{}

\begin{document}
\maketitle
\begin{abstract} 
        Studies about the set of positive equilibria ($E_+$) of kinetic systems have been focused on mass action, and not that much on power law kinetic (PLK) systems, even for PL-RDK systems (PLK systems where two reactions with identical reactant complexes have the same kinetic order vectors).
        For mass action, reactions with different reactants have different kinetic order rows. A PL-RDK system satisfying this property is called factor span surjective (PL-FSK).
		In this work, we show that a cycle terminal PL-FSK system with $E_+\ne \varnothing$ and has independent linkage classes (ILC) is a poly-PLP system, i.e., $E_+$ is the disjoint union of log-parametrized sets.
		The key insight for the extension is that factor span surjectivity induces an isomorphic digraph structure on the kinetic complexes. The result also completes, for ILC networks, the structural analysis of the original complex balanced generalized mass action systems (GMAS) by M\"uller and Regensburger.
		We also identify a large set of PL-RDK systems where non-emptiness of $E_+$ is a necessary and sufficient condition for non-emptiness of each set of positive equilibria for each linkage class. These results extend those of Boros on mass action systems with ILC.
        We conclude this paper with two applications of our results. Firstly, we consider absolute complex balancing (ACB), i.e., the property that each positive equilibrium is complex balanced, in poly-PLP systems.
		Finally, we use the new results to study absolute concentration robustness (ACR) in these systems. In particular, we obtain a species hyperplane containment criterion to determine ACR in the system species.
	
\end{abstract}

\thispagestyle{empty}
	\section{Introduction}
	
	
	A power law kinetic system $(\mathcal{N}, K)$ consists of a chemical reaction network $\mathcal{N}$ and a kinetics $K$ which assigns to each reaction $q$ a function $K_q$ of the form 
	$K_q(x) = {k_q}\prod\limits_{i=1}^{m} {{x_i}^{{F_{qi}}}} $ for $q = 1,\ldots,r$
	with $k_q \in \mathbb{R}_{>0}$, $F_{qi} \in \mathbb{R}$ and $x \in \Omega$, a subset of $\mathbb{R}^m_{\ge 0}$ containing $\mathbb{R}^m_{> 0}$. The $r \times m$ matrix $[F_{qi}]$ is called the system's kinetic order matrix $F$ and each $k_q$ is a rate constant. Mass action systems can be viewed as a subset of the power law kinetic (PLK) systems, with the kinetic order matrix consisting of the stoichiometric coefficients of each reaction's reactant complex. PLK systems in which, as in mass action systems, branching reactions of a reactant complex have identical rows in the kinetic order matrix form the subset of PL-RDK systems (RDK = reactant determined kinetic orders) systems, its (non-empty) complement is denoted by PL-NDK. A mass action system also has the additional property that reactions with different reactants have different kinetic order rows -- a PL-RDK system with this property is called a factor span surjective (denoted by PL-FSK) system.
	
	The study of power law kinetic systems in Chemical Reaction Network Theory (CRNT) was revived by S. M\"uller and G. Regensburger in 2012 with the novel, more geometric concepts of generalized mass action systems (GMAS) and kinetic deficiency \cite{MULL2012}. In a GMAS, the rows of the kinetic order matrix of a PLK system $(\mathcal{N}, K)$ are called kinetic complexes and collected into a set $\widetilde{\mathcal{C}}$, which for a weakly reversible network $\mathcal{N}$ is the image of a bijective map $\widetilde{\phi}:\mathcal{C} \to \widetilde{\mathcal{C}}$. The requirement of $\widetilde{\phi}$ being a surjective map is equivalent to the system being PL-RDK and injectivity to its factor span surjectivity (i.e. in PL-FSK). They showed that complex balanced GMAS share many properties with their mass action counterparts including the log parametrization of their complex balanced equilibria, i.e., their set $Z_+(\mathcal{N}, K)= \{ x \in \mathbb{R}^{\mathcal{S}}_{>0} | \log (x) - \log (x^*) \in \widetilde{S}^\perp \}$, where $x^*$ is a given complex balanced equilibrium and $\widetilde{S}$ is the system's kinetic order subspace (the kinetic analogue of the stoichiometric subspace $S$). In \cite{MURE2014}, they expanded the GMAS concept by dropping the injectivity requirement for $\widetilde{\phi}$ and extended their results to complex balanced PL-RDK systems.
	
	In the past decade since the work of M\"uller and Regensburger, numerous advances have been made in the study of positive equilibria of non-mass action PLK systems (e.g., \cite{talabis:positive:eq:jomc}, \cite{METJ2018}, \cite{TMJ2019}, \cite{CMPY2019}, \cite{FLRM2020}, \cite{FOME2021}). However, they have also exclusively focused on the subset of complex balanced equilibria. Even the results of Talabis et al. \cite{talabis:positive:eq:jomc} and Fortun et al. \cite{FLRM2020} on the set of (all) positive equilibria $E_+(\mathcal{N}, K)$ of weakly reversible PLK systems in their  Deficiency Zero and Deficiency One Theorems, can be viewed as results on complex balancing since in these cases, $E_+(\mathcal{N}, K) = Z_+(\mathcal{N}, K)$ \cite{JTAM2021}. To our knowledge, new results have been published for non-weakly reversible ``$\widehat{T}$-rank maximal'' kinetic (PL-TIK) systems satisfying the conditions of Deficiency One Theorem and weakly reversible PL-NIK systems, i.e., all kinetic orders are non-negative, on conservative and concordant networks, as a special case of a result of Shinar and Feinberg \cite{SHFE2012} on weakly monotonic kinetic systems ($\widehat{T}$ is the augmented matrix of kinetic complexes where the non-reactant columns are deleted). The goal of this paper is to contribute two new results in this regard.
	
	A network is cycle terminal if each of its complexes is a reactant complex.
	Our first main result establishes that any cycle terminal PL-FSK system with $E_+(\mathcal{N}, K) \ne \varnothing$ and has independent linkage classes (ILC) is a poly-PLP  system, i.e., $E_+(\mathcal{N}, K)$ is the disjoint union of sets of the form $\{ x \in \mathbb{R}^\mathcal{S}_{>0} | \log (x) - \log (x^*) \in \widetilde{S}^{\perp} \}$, where $x^*$ are given equilibria and the union, when finite has $| E_+(\mathcal{N}, K) \cap Q|$ terms. This result extends the results of Boros in his PhD thesis \cite{boros:thesis} on mass action systems with ILC -- in fact, many of our proofs follow his arguments closely. The key insight for the extension is that factor span surjectivity induces an isomorphic digraph structure on the kinetic complexes, and hence a bijection between their decompositions. This result also completes, for ILC networks, the structural analysis of the original complex balanced GMAS by M\"uller and Regensburger focused on the subset $Z_+(\mathcal{}N, K )$ of complex balanced equilibria. Furthermore, it identifies the class of systems which includes the counterexample of Jose et al. \cite{JTAM2021} for the extension of the theorem of Horn-Jackson on absolutely complex balanced (ACB) systems to complex balanced PL-RDK systems \cite{JTAM2021}.
	
	It follows from the Feinberg's Decomposition Theorem that for a network with ILC and any kinetics, $E_+(\mathcal{N},K) \ne \varnothing \implies  E_+(\mathcal{L}, K) \ne \varnothing$ for each linkage class $\mathcal{L}$. Our second main result identifies a set of PL-RDK systems for which the condition is also sufficient, i.e., $E_+(\mathcal{L}, K) \ne \varnothing$ for each linkage class $\mathcal{L}$ implies $E_+(\mathcal{N},K) \ne \varnothing$. This result also extends a result of Boros on mass action systems with ILC. Furthermore, it extends, for networks with ILC, results of Talabis et al. \cite{talabis:positive:eq:jomc} on PL-TIK systems to its superset of $\widehat{T}$-independent PL-RDK systems.
	
	We conclude the paper with two applications of our results. In the first application, we consider absolute complex balancing in poly-PLP systems. A complex balanced kinetic system, i.e., $Z_+(\mathcal{N},K) \ne \varnothing$, is absolutely complex balanced (ACB) if  $E_+(\mathcal{N},K)=Z_+(\mathcal{N},K)$, i.e., every positive equilibrium is complex balanced. In 1972, F. Horn and R. Jackson derived the following fundamental result: for mass action kinetics, any complex balanced system is absolutely complex balanced. In \cite{JTAM2021}, Jose et al. posed the following question: to which other kinetic systems, besides mass action, does the Horn-Jackson result extend? Using the results of Jose et al., we show that for complex balanced PL-RDK systems which are poly-PLP with flux subspace $\widetilde{S}$ (the kinetic order subspace of the system), the multi-PLP systems are precisely those which are not ACB. We then apply our results for cycle terminal PL-FSK systems to show that for these systems, multi-PLP is equivalent to multi-stationarity.
	This criterion identifies the set of monostationary PL-FSK systems with ILC as a new class of PL-RDK systems to which the Horn-Jackson result extends.
	In addition, we apply our results to study absolute concentration robustness (ACR) in these systems. The property of ACR in a species was introduced by G. Shinar and M. Feinberg in a well-known paper in the journal {\it{Science}} and denotes the invariance of the species concentration at all positive equilibria. Their results for mass action systems have been extended to power law kinetic systems, and we use the specific structure of the positive equilibria sets of the systems studied to obtain a species hyperplane containment criterion to determine ACR in the system species.
	
	The paper is organized as follows: Section \ref{prelim} collects the basic concepts and results on reaction networks and kinetic systems needed in the later sections. In Section \ref{results1}, poly-PLP systems are introduced and the first main result is shown. $\widehat{T}$-independence is discussed in Section \ref{results2} and shown to ensure the sufficiency of non-empty linkage class equilibria sets for the same in the whole network for networks with ILC.
	In Section \ref{results4} we study absolute complex balancing in poly-PLP systems and identify a new class of poly-PLP PL-RDK systems with the property.
	Section \ref{results3} formulates a general species hyperplane criterion for ACR and applies it to the systems studied. Summary and outlook are provided in Section \ref{summary}.

	\section{Preliminaries}
	\label{prelim}
	In this section, we provide essential background for our succeeding discussions. We present a number of definitions and useful results on chemical reaction networks and chemical kinetic systems primarily from the works of M. Feinberg \cite{Feinberg:1987:CRN,feinberg:book,Feinberg:1979:lec,Feinberg:1995:existence}, F. Horn and R. Jackson \cite{HORN1972:GMAK}, and C. Wiuf and E. Feliu \cite{WIUF2013:PLK}.
	
	\subsection{Fundamentals of chemical reaction networks}
	We start by giving a formal definition of a chemical reaction network, or simply CRN:
	\begin{definition}
		A {\bf chemical reaction network} is a triple $\mathcal{N} := \left(\mathcal{S},\mathcal{C},\mathcal{R}\right)$ of nonempty finite sets $\mathcal{S}$, $\mathcal{C} \subseteq \mathbb{R}_{\ge 0}^\mathcal{S}$, and $\mathcal{R} \subset \mathcal{C} \times \mathcal{C}$, of $m$ species, $n$ complexes, and $r$ reactions, respectively, that satisfies the following properties:
		\begin{itemize}
		\item[i.] $y \to y \notin \mathcal{R}$ for each $y \in \mathcal{C}$, and
		\item[ii.] for each $y \in \mathcal{C}$, there exists $y' \in \mathcal{C}$ such that $y \to y' \in \mathcal{R}$ or $y' \to y \in \mathcal{R}$.
		\end{itemize}
	\end{definition}

Chemical reaction networks can be treated as directed graphs, which are known as reaction graphs in literature, where complexes are vertices and reactions are arcs.
The (strongly) connected components are precisely the {\bf (strongly) linkage classes} of the CRN. A strong linkage class is a {\bf terminal strong linkage class} if there is no reaction from a complex in the strong linkage class to a complex outside the given strong linkage class.

For a given CRN with $\ell, sl,$ and $t$ are
the number of linkage classes, the number of strong linkage classes, and the number of terminal strong linkage classes, respectively, we have $\ell \le t \le sl$.
A CRN is {\bf weakly reversible} if $sl=\ell$, i.e., each linkage class is strongly connected, and is {\bf $t$-minimal} if $t=\ell$.
The terminal strong linkage classes can be of two kinds: cycles (in the sense of directed graphs) that are not necessarily simple, and singletons, which we call {\bf{terminal points}}.
We let $n_r$ be the number of reactant complexes of a CRN. Then, $n-n_r$ is the number of terminal points. A CRN is {\bf cycle terminal} if and only if $n - n_r = 0$.

We then proceed with introducing the following matrices that are essential to our work:	
	\begin{definition}
		The {\bf molecularity matrix} or {\bf matrix of complexes} $Y$ is an $m\times n$ matrix such that $Y_{ij}$ is the stoichiometric coefficient of species $i$ in a complex $j$.
		The {\bf incidence matrix} $I_a$ is an $n\times r$ matrix such that 
		$${\left( {{I_a}} \right)_{ij}} = \left\{ \begin{array}{rl}
			- 1&{\rm{ if \ complex \ }} i {\rm{ \ is\ the\ reactant \ complex \ of \ reaction \ }}{j},\\
			1&{\rm{  if \ complex \ }}{i}{\rm{ \ is \ the\ product \ complex \ of \ reaction \ }}{j},\\
			0&{\rm{    otherwise}}.
		\end{array} \right.$$
		The {\bf stoichiometric matrix} $N$ is the $m\times r$, which is given by the product $Y\cdot I_a$.
	\end{definition}

We now provide definitions for what we call the stochiometric subspace and the deficiency of a chemical reaction network.
	
	\begin{definition}
		The {\bf stoichiometric subspace} of the network $\mathcal{N}$ is defined as $$S := \left\{{y' - y \in \mathbb{R}^\mathcal{S}| y \to y' \in \mathcal{R}}\right\},$$
		i.e., the linear span of the reaction vectors over $\mathbb{R}$. The {\bf rank} of the network is given by $s:=\dim S$.
		For $x \in \mathbb{R}_{ > 0}^\mathcal{S}$, its {\bf stoichiometric compatibility class} is defined as
		$\left( {x + S} \right) \cap \mathbb{R}_{ \ge 0}^\mathcal{S}$.
		Two vectors $x^{*}, x^{**} \in {\mathbb{R}^m}$ are {\bf stoichiometrically compatible} if $x^{*}-x^{**}$ is an element of the stoichiometric subspace $S$.
	\end{definition}

	\begin{definition}
	The {\bf deficiency} of a CRN is $\delta=n-\ell-s$ where $n$ is the number of complexes, $\ell$ is the number of linkage classes, and $s$ is the rank of the network.
	\end{definition}

Decomposition of a chemical reaction network is induced from partitioning the reaction set into subsets. M. Feinberg introduced the so-called independent decomposition of chemical reaction networks as given in Definition \ref{def:independent:dec} (Section 5.4 in \cite{Feinberg:1987:CRN}, Appendix 6.A in \cite{feinberg:book}). In our work, however, we consider the linkage classes as the subnetworks under network decomposition.

\begin{definition}
	A decomposition of a chemical reaction network $\mathcal{N}$ into $k$ subnetworks of the form $\mathcal{N}=\mathcal{N}_1 \cup \cdots \cup \mathcal{N}_k$  is {\textbf{independent}} if its stoichiometric subspace is equal to the direct sum of the stoichiometric subspaces of its subnetworks, i.e., $S=S_1 \oplus  \cdots  \oplus S_k$.
	\label{def:independent:dec}
\end{definition}

\begin{remark}
We emphasize that whenever we write equations like $\delta = \delta ^1 + \cdots + \delta ^\ell$ or $S = S^1 \oplus \cdots \oplus S^\ell$ as in Proposition \ref{directsum:deficiencies}, the $i$th linkage class has deficiency $\delta^i$ and stoichiometric subspace $S^i$.
A reaction network that satisfies property (i) or property (ii) in Proposition \ref{directsum:deficiencies} possesses the {\bf{independent linkage classes}} property, or simply {\bf{ILC}}.
\label{rem:ILC}
\end{remark}
	
	\begin{proposition} \cite{boros:paper:thesis,boros:thesis}
	Let $\mathcal{N}$ be a reaction network.
	Then the following statements are equivalent.
	\begin{itemize}
		\item[i.] $\delta = \delta ^1 + \cdots + \delta ^\ell$
		\item[ii.] $S = S^1 \oplus \cdots \oplus S^\ell$
	\end{itemize}
	\label{directsum:deficiencies}
	\end{proposition}

These next two results were taken from Boros \cite{boros:paper:thesis,boros:thesis} that we will use in the proofs of our main results. The first part of Lemma \ref{unique:vector:subspace} appears in various works in literature such as in Corollary 4.14 of \cite{Feinberg:1979:lec} and in Section 4 of \cite{HORN1972:GMAK}. On the other hand, the next lemma is a consequence of the famous Farkas' Lemma.

	\begin{lemma}
		Let $\mathcal S$ be a linear subspace of $\mathbb{R}^m$, $\mathcal{P}=(p+\mathcal{S}) \cap \mathbb{R}^m_{\ge 0}$ for some $p \in \mathbb{R}^m_{>0}$, and fix $x^* \in \mathbb{R}^m_{> 0}$. Then
		\begin{itemize}
			\item[i.] there exists a unique $x \in \mathcal{P} \cap \mathbb{R}^m_{> 0}$ such that ${\log \left( x \right) - \log \left( x^* \right)} \in \mathcal{S}^{\perp}$, and
			\item[ii.] the set $\{x \in \mathbb{R}_{ > 0}^m | {\log \left( x \right) - \log \left( x^* \right)} \in \mathcal{S}^{\perp}\}$ is $C^{\infty}$-diffeomorphic to $\mathbb{R}^{n-{\rm{rank}}\mathcal{S}}$, and hence, connected.
		\end{itemize}
		\label{unique:vector:subspace}
	\end{lemma}

	\begin{lemma}
		Let $\ell$, $m$, and $c^1, \ldots, c^\ell$ be positive integers. Let $A_i \in {\mathbb{R}}^{{c^i}\times m}$ and $b_i \in {\mathbb{R}}^{c^i}$ for $i=1,\ldots,\ell$. Suppose further that
		\begin{itemize}
		\item[i.] $\{x \in \mathbb{R}^m| A_ix=b_i\} \ne \varnothing$ for each $i=1,\ldots,\ell$, and
		\item[ii.] ${\mathop{\rm Im}\nolimits} \left\{ {A_1^{\top},\ldots,A_\ell^\top} \right\} = {\mathop{\rm Im }\nolimits} A_1^\top \oplus \cdots \oplus {\mathop{\rm Im }\nolimits} A_\ell^\top$.
		\end{itemize}
	Then
	$\bigcap\limits_{i = 1}^\ell {\left\{ {x \in {\mathbb{R}^m}|{A_i}x = {b_i}} \right\} \ne \varnothing}.$
	\label{solution:system:ran}
	\end{lemma}

\subsection{Fundamentals of chemical kinetic systems}
We now introduce our definitions for a chemical kinetic system and the species formation rate function. Here, we associate a kinetics for a given chemical reaction network to describe its dynamical properties.

\begin{definition}
A {\bf kinetics} for a reaction network $(\mathcal{S}, \mathcal{C}, \mathcal{R})$ is an assignment to
each reaction $y \to y' \in \mathcal{R}$ of a continuously differentiable {\bf rate function} $K_{y\to y'}: \mathbb{R}^\mathcal{S}_{>0} \to \mathbb{R}_{\ge 0}$ such that the following positivity condition holds:
$K_{y\to y'}(c) > 0$ if and only if ${\rm{supp \ }} y \subset {\rm{supp \ }} c$.
The system $\left(\mathcal{N},K\right)$ is called a {\bf chemical kinetic system}.
\end{definition}


\begin{definition}
	The {\bf species formation rate function} (SFRF) of a chemical kinetic system $(\mathcal{N},K)$ is defined as $f\left( x \right) = NK(x)= \displaystyle \sum\limits_{{y} \to {y'} \in \mathcal{R}} {{K_{{y} \to {y'}}}\left( x \right)\left( {{y'} - {y}} \right)}.$
\end{definition}
The dynamical system or system of ordinary differential equations (ODEs) of a chemical kinetic system is given by $\dfrac{{dx}}{{dt}} = f\left( x \right)$. An {\bf equilibrium} or a {\bf steady state} is a zero of $f$.

\begin{definition}
	The {\bf set of positive equilibria} of a chemical kinetic system $\left(\mathcal{N},K\right)$ is given by $${E_ + }\left(\mathcal{N},K\right)= \left\{ {x \in \mathbb{R}^m_{>0}|f\left( x \right) = 0} \right\}.$$ We also denote this set by $E_ +$ for simplicity.
\end{definition}

A chemical reaction network is said to admit {\bf multiple positive equilibria} if there exists a set of positive rate constants such that the system of ODEs admits more than one stoichiometrically compatible equilibria.
\begin{definition}
The {\bf set of complex balanced equilibria} of a chemical kinetic system $\left(\mathcal{N},K\right)$ is given by 
\[{Z_ + }\left(\mathcal{N},K\right) = \left\{ {x \in \mathbb{R}_{ > 0}^m|{I_a} \cdot K\left( x \right) = 0} \right\} \subseteq {E_ + }\left(\mathcal{N},K\right).\]
\end{definition}
A positive vector $c \in \mathbb{R}^m$ is complex balanced if $K\left( c \right)$ is contained in $\ker{I_a}$. A chemical kinetic system is {\bf{complex balanced}} if it has a complex balanced equilibrium.

We now consider an important class of chemical kinetics that generalizes the well-known mass action.
\begin{definition}
	A kinetics $K$ is a {\bf power law kinetics} (PLK) if 
	${K_i}\left( x \right) = {k_i}\prod\limits_j {{x_j}^{{F_{ij}}}}  := {k_i}{{x^{{F_{i}}}}} $  $\forall$ reaction $i =1,\ldots,r$ where ${k_i} \in {\mathbb{R}_{ > 0}}$ and ${F_{ij}} \in {\mathbb{R}}$. The $r \times m$ (reaction by species) matrix $F=\left[ F_{ij} \right]$, containing the kinetic order values, is called the {\bf kinetic order matrix}, and $k \in \mathbb{R}^r$ is called the {\bf rate vector}.
	\label{def:power:law}
\end{definition}

If the kinetic order matrix contains the corresponding stoichiometric coefficients of reactant $y$ for each reaction $y \to y'$ in the network, then  then the system follows the {\bf mass action kinetics}.


\begin{definition}
	A PLK system has {\bf reactant-determined kinetics} (of type PL-RDK) if for any two reactions $i, j$ with identical reactant complexes, the corresponding rows of kinetic orders in $F$ are identical, i.e., ${F_{ik}} = {F_{jk}}$ for $k = 1,2,...,m$. On the other hand, it has a {\bf non-reactant-determined kinetics} (of type PL-NDK) if there exist two reactions with the same reactant complexes whose corresponding rows in $F$ are not identical.
\end{definition}

Fig. \ref{fig:NDK:RDK} shows that power law kinetics can either be PL-RDK or PL-NDK. In addition, mass action kinetics are PL-RDK.

	\begin{figure}
	\centering
		\includegraphics[width=8cm,height=5cm,keepaspectratio]{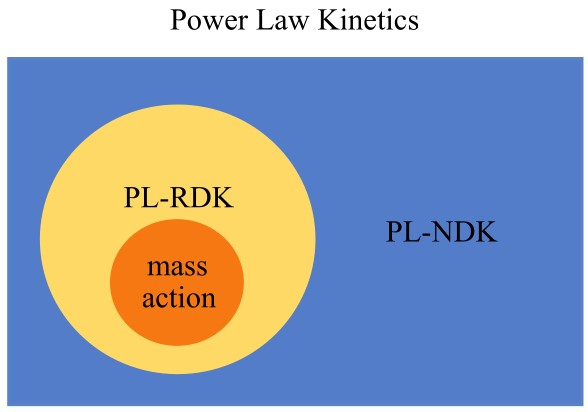}
		\caption{The sets PL-NDK and PL-RDK are complementary in the set of power law kinetics, and mass action kinetics are PL-RDK.}
		\label{fig:NDK:RDK}
\end{figure}
	

From Definition \ref{def:power:law} and due to S. M\"uller and G. Regensburger \cite{MURE2014}, we present the following definitions:

\begin{definition}
	The $m \times n$ matrix $\widetilde{Y}$ is defined as
	$${\left( {\widetilde{Y}} \right)_{ij}} = \left\{ \begin{array}{l}
		{F_{ki}}{\text{  \ \ \ \ \ \ \  if \ $j$ {\text{is the reactant complex in reaction }}$k$}}\\
		{0} {\text{  \ \ \ \ \ \ \  \ \   otherwise \ }}
	\end{array} \right..$$
\end{definition}

\begin{definition}
	The $m \times n_r$ T-matrix is the truncated $\widetilde{Y}$ where the non-reactant columns are deleted.
\end{definition}

\begin{definition}
	The block matrix $\widehat{T} \in \mathbb{R}^{(m+l) \times n_r}$ is defined as
	$\widehat{T} = \left[ {\begin{array}{*{20}{c}}
			T\\
			{{L^\top}}
	\end{array}} \right]$
where the $n_r \times \ell$ matrix $L = [e_1, e_2, \ldots , e_\ell]$, and $e_i$ is a characteristic vector for the linkage class $\mathcal{L}_i$.
\end{definition}
A PL-TIK system is a PL-RDK system whose block matrix has maximal column rank. The subset of PL-TIK with linear independent T-matrix columns is denoted by PL-RLK (reactant set linear independent power-law kinetics).
	
	\section{Positive equilibria of PL-FSK systems on cycle terminal networks with ILC}
	\label{results1}
	
	In this section, we determine the structure of the positive equilibria set of factor span surjective power law kinetic (PL-FSK) systems on cycle terminal networks with independent linkage classes (ILC). These systems include the complex balanced GMAS where M\"uller and Regensburger determined the structure of the subset of complex balanced equilibria. We show that they are poly-PLP systems, i.e., $E_+(N, K)$ is the disjoint union of sets of the form $\{ x \in \mathbb{R}^\mathcal{S}_{>0} | \log (x) - \log (x^*) \in \widetilde{S}^\perp \}$, where $x^*$ are given equilibria.
	Note that the kinetic
	order subspace $\widetilde{S}$ is the kinetic analogue of the stoichiometric subspace $S$, which is generated by
	the differences of kinetic complexes.
	Our results partially extend the result of B. Boros on mass action systems with ILC and those of Jose et al. on (mono-) PLP systems.
	
		
		We now introduce the following definition of an LP set from Jose et al. \cite{JTAM2021} and the definition of a poly-PLP system.
		
		\begin{definition}
		An {\bf{LP set}} is a non-empty subset of $\mathbb{R}^\mathcal{S}_{>0}$
		of the form $E(P,x^*)$ where $P$ is a subspace of $\mathbb{R}^\mathcal{S}$, which is called an LP set's {\underline{flux subspace}}, and $x^*$ a given element of $\mathbb{R}^\mathcal{S}_{>0}$, which is called the LP set's {\underline{reference point}}.
		$P^\perp$ is called an LP set's {\underline{parameter subspace}}, and
		the positive cosets of $P$ are called LP set's {\underline{flux classes}}.
		\end{definition}
		
		\begin{definition}
		A kinetic system $(\mathcal{N}, K)$ is a {\bf{poly-PLP}} system if its set of positive equilibria $E_+(\mathcal{N}, K)$ is the disjoint union of LP sets with flux subspace $P_E$ and reference points $\{x_i^* | i =1,\ldots,\mu\}$ where $\mu := |E_+ \cap Q|$ (with $Q$ as a flux class) or $\infty$ is the coset intersection count.  Analogously, a system is {\bf{poly-CLP}} if its set of complex balanced equilibria $Z_+(\mathcal{N}, K)$ is the disjoint union of LP sets.  A poly-PLP (poly-CLP) system is called {\bf multi-PLP (multi-CLP)} if $\mu \ge 2$, otherwise {\bf PLP (CLP)}. We also denote a poly-PLP (poly-CLP) system with $\mu$-PLP ($\mu$-CLP) if we want to emphasize the value of $\mu$.
		\end{definition}
		
		An example of a poly-PLP system is given in Example \ref{ex:poly:PLP}. We use our main result to show that it is indeed such system.
		
	
	We now introduce the key insight for the extension of the results of Boros in his PhD thesis \cite{boros:paper:thesis,boros:thesis} from mass action to a large class of power law kinetics.
	Note that when a PL-RDK (kinetics) is PL-FSK (power law kinetics which is factor span surjective), different reactant complexes have different kinetic complexes.
	Factor span surjectivity for PL-RDK induces an isomorphic digraph structure on the kinetic complexes as described in Proposition \ref{cycle:terminal:PLFSK:iso}.
	We first state the following definition:
	
	\begin{definition}
		Let $(\mathcal{N}, K)$ be a cycle terminal PLK system. The set of kinetic complexes of a complex $y$ is defined as $\mathcal{\widetilde{C}}(y) := \{ F_q | q \in \mathcal{R}(y)\}$, where
		$\mathcal{R}(y)$ is the set of (branching) reactions of complex $y$.
		Let $\widetilde{y}$ denote an element of the set of kinetic complexes of $y$. For any reaction $q: y \to y'$, $$\mathcal{\widetilde{R}}(q) := \{ \widetilde{y} \to \widetilde{y'}|
		\widetilde{y} \in \mathcal{\widetilde{C}}(y), \widetilde{y'} \in \mathcal{\widetilde{C}}(y')\}$$ is the set of kinetic reactions of $q$, and $\widetilde{q}$ denotes an element of the set.
		\label{def:cycle:term}
	\end{definition}
	
	We now define the reaction network $\mathcal{\widetilde{N}}$ of kinetic complexes of a cycle terminal PLK system $(\mathcal{\widetilde{N}}, K)$:
	\begin{definition}
		The set of kinetic complexes has a reaction network structure given by $\mathcal{\widetilde{N}}= (\mathcal{S}, \mathcal{\widetilde{C}}, \mathcal{\widetilde{R}})$
		with $\mathcal{\widetilde{C}}= \bigcup\limits_y   \mathcal{\widetilde{C}} (y)$ 
		and $\mathcal{\widetilde{R}}= \bigcup\limits_q   \mathcal{\widetilde{R}} (q)$.
		\label{def:cyc:terminal:network}
	\end{definition}
	
	
For cycle terminal PL-RDK systems, S. M\"uller and G. Regensburger introduced an analogue of the stoichiometric subspace of $\mathcal{N}$ for $\widetilde{\mathcal{N}}$:

\begin{definition}
The {\bf kinetic order subspace} of a cycle terminal PL-RDK system $(\mathcal{N},K)$ is defined as the linear span $\widetilde{S} := < \widetilde{y'}- \widetilde{y} | y \to y' \in \mathcal{R}>$. It is the stoichiometric subspace of $\widetilde{\mathcal{N}}$. The {\bf kinetic deficiency} $\widetilde{\delta}$ is defined as
$\widetilde{\delta}=n-l-\widetilde{s}$, where $\widetilde{s}={\rm{dim \ }} \widetilde{S}$, i.e., the rank of the network $\widetilde{\mathcal{N}}$.
\end{definition}
	
	\begin{figure}
	\centering
		\includegraphics[width=8cm,height=5cm,keepaspectratio]{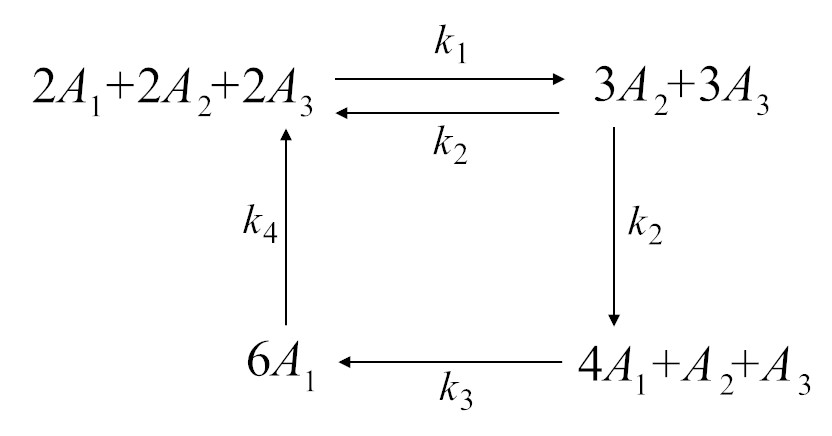}
		\caption{A chemical reaction network from Jose et al. \cite{JTAM2021}}
		\label{fig:Josecounterexample}
\end{figure}
	
	\begin{example}
	Consider the example of Jose et al. \cite{JTAM2021}, with underlying network $\mathcal{N}$, depicted in Fig. \ref{fig:Josecounterexample}. We let the complexes in $\mathcal{N}$ be labeled as $C_1=2A_1+2A_2+2A_3$, $C_2=3A_2+3A_3$, $C_3=4A_1+A_2+A_3$, and $C_4=6A_1$. Then, we have these reactions:
	\begin{align*}
	    r_1&:={C}_1 \to {C}_2\\
	    r_2&:={C}_2 \to {C}_1\\
	    r_3&:={C}_2 \to {C}_3\\
	    r_4&:={C}_3 \to {C}_4\\
	    r_5&:={C}_4 \to {C}_1\\
	\end{align*}
	In addition, the kinetics is defined using the following kinetic order matrix \cite{JTAM2021}:
	\[F = \left[ {\begin{array}{*{20}{c}}
0&{ - 1}&1\\
{ - 1}&{ - 1}&1\\
{ - 1}&{ - 1}&1\\
0&{ - 2}&0\\
0&0&{ - 2}
\end{array}} \right].\]

	Note that each complex in $\mathcal{N}$ is a reactant complex, and hence, the network is cycle terminal.
	Then, we have the following kinetic complexes: 
	$\widetilde{C}_1:=-A_2+A_3$, $\widetilde{C}_2:=-A_1-A_2+A_3$, $\widetilde{C}_3:=-2A_2$, and $\widetilde{C}_4:=-2A_3$.
	One may write the kinetic complexes by expressing the species as the standard basis vectors of ${\mathbb{R}}^m$, for instance, $-A_2+A_3$ may be written as $(0,-1,1)$. Moreover, we form the following kinetic reactions:
	\begin{align*}
	    \mathcal{R}(r_1)&=\{\widetilde{C}_1 \to \widetilde{C}_2\}\\
	    \mathcal{R}(r_2)&=\{\widetilde{C}_2 \to \widetilde{C}_1\}\\
	    \mathcal{R}(r_3)&=\{\widetilde{C}_2 \to \widetilde{C}_3\}\\
	    \mathcal{R}(r_4)&=\{\widetilde{C}_3 \to \widetilde{C}_4\}\\
	    \mathcal{R}(r_5)&=\{\widetilde{C}_4 \to \widetilde{C}_1\}\\
	\end{align*}
	We now form the structure 
	$\mathcal{\widetilde{N}}= (\mathcal{S}, \mathcal{\widetilde{C}}, \mathcal{\widetilde{R}})$
		with $\mathcal{\widetilde{C}}= \{\widetilde{C}_1, \widetilde{C}_2, \widetilde{C}_3, \widetilde{C}_4 \}$ 
		and
		$\mathcal{\widetilde{R}}=\{\widetilde{C}_1 \to \widetilde{C}_2,\widetilde{C}_2 \to \widetilde{C}_1, \widetilde{C}_2 \to \widetilde{C}_3, \widetilde{C}_3 \to \widetilde{C}_4, \widetilde{C}_4 \to \widetilde{C}_1\}$.
	\end{example}
	
	The following proposition
	is a key insight for our results in this section. In particular, we need the assumption of a kinetic system to be cycle terminal and PL-FSK for digraph isomorphism between the network $\mathcal{N}$ and the induced network $\widetilde{\mathcal{N}}$ as defined in Definition \ref{def:cyc:terminal:network} to proceed with Lemma \ref{pi:eq} and succeeding propositions.
	
	\begin{proposition}
		Let $(\mathcal{N}, K)$ be a cycle terminal PL-FSK system. Then
		\begin{itemize}
			\item[i.] $\widetilde{\mathcal{N}}$ is isomorphic (as a digraph) to $\mathcal{N}$.  Consequently, $I_a = {\widetilde{I_a}}$,
			and ${\rm{Im \ }} I_a$ is isomorphic to ${\rm{Im \ }} \widetilde{I_a}$.
			\item[ii.] The isomorphism induces a bijection between the set of decompositions of $\mathcal{N}$ and $\widetilde{\mathcal{N}}$.
		\end{itemize}
		\label{cycle:terminal:PLFSK:iso}
	\end{proposition}
	
	Before we present the main results, we first present these basic results and definitions that are essential in the latter part of this section. We define $\pi$ (Definition \ref{def:pi} given below) in the sense of PL-RDK so that we are ensured that $\pi$ is indeed a function. We use Lemma \ref{pi:eq} to establish Lemma \ref{log:subset:equilibria}, which gives a parametrization of a non-empty subset of the set of equilibria $E_+$ of a cycle terminal PL-FSK system that satisfies the ILC property, i.e., $\delta=\delta_1+\cdots+\delta_\ell$ where $\ell$ is the number of linkage classes of the underlying network, with provision of course of the non-emptiness of $E_+$.
	
	\begin{proposition}
		Let $(\mathcal{N}, K)$ be a kinetic system that satisfies $\delta=\delta_1+\cdots+\delta_\ell$. Suppose that $x \in \mathbb{R}^m_{\ge0}$. Then $f(x)=0$ if and only if $f^i(x)=0$ for each $i=1,\ldots,\ell$.
		\label{SFRF:linkageclass}
	\end{proposition}

	\begin{proof}
	Note that $f(x) \in S$, $f^1(x) \in S^1$, \dots, $f^\ell(x) \in S^\ell$. The conclusion follows from Proposition \ref{directsum:deficiencies}.
	\end{proof}
	
	\begin{definition}
		Let $(\mathcal{N},K)$ be a PL-RDK system. For each
		complex $y$, the function ${\pi _y}:\mathbb{R}_{ > 0}^m \times \mathbb{R}_{ > 0}^m \to \mathbb{R}_{ > 0}$ is given by
		\[{\pi _y}\left( {x,x'} \right) = \prod\limits_{s = 1}^m {{{\left( {\frac{{{x_s}}}{{x{'_s}}}} \right)}^{{{\widetilde{Y}}_{sy}}}}}. \]
		\label{def:pi}
	\end{definition}
	
	\begin{lemma}
		Let $(\mathcal{N},K)$ be a cycle terminal PL-FSK system, and the  function $\pi_y$ be given in Definition \ref{def:pi}.
		Then, the following statements are equivalent.
		\begin{itemize}
			\item[i.] ${\pi _y}\left( {x,x'} \right) = {\pi _{y'}}\left( {x,x'} \right)$ for any $y \to y' \in \mathcal{R}$
			\item[ii.] $\left\langle {{{\widetilde{Y}}_{y'}} - {{\widetilde{Y}}_{y}},\log \left( x \right) - \log \left( x' \right)} \right\rangle  = 0$ for any $y \to y' \in \mathcal{R}$
		\end{itemize}
		\label{pi:eq}	
	\end{lemma}
	\begin{proof}
		\[{\pi _y}\left( {x,x'} \right) = {\pi _{y'}}\left( {x,x'} \right) \Leftrightarrow \prod\limits_{s = 1}^m {{{\left( {\frac{{{x_s}}}{{x{'_s}}}} \right)}^{{{\widetilde{Y}}_{sy}}}}}  = \prod\limits_{s = 1}^m {{{\left( {\frac{{{x_s}}}{{x{'_s}}}} \right)}^{{{\widetilde{Y}}_{sy'}}}}} \]
		\[ \Leftrightarrow {{\widetilde{Y}}_{1y}}\log \left( {\frac{{{x_1}}}{{x{'_1}}}} \right) + \cdots + {{\widetilde{Y}}_{my}}\log \left( {\frac{{{x_m}}}{{x{'_m}}}} \right) = {{\widetilde{Y}}_{1y'}}\log \left( {\frac{{{x_1}}}{{x{'_1}}}} \right) + \cdots + {{\widetilde{Y}}_{my'}}\log \left( {\frac{{{x_m}}}{{x{'_m}}}} \right)\]
		\[ \Leftrightarrow \left( {{{\widetilde{Y}}_{1y'}} - {{\widetilde{Y}}_{1y}}} \right)\log \left( {\frac{{{x_1}}}{{x{'_1}}}} \right) + \cdots + \left( {{{\widetilde{Y}}_{my'}} - {{\widetilde{Y}}_{my}}} \right)\log \left( {\frac{{{x_m}}}{{x{'_m}}}} \right) = 0\]
		\[ \Leftrightarrow \left\langle {{{\widetilde{Y}}_{y'}} - {{\widetilde{Y}}_{y}},\log \left( {\frac{x}{{x'}}} \right)} \right\rangle  = 0\]	
	\end{proof}
	
	\begin{lemma}
		Let $(\mathcal{N},K)$ be a cycle terminal PL-FSK system that satisfies $\delta=\delta_1+\cdots+\delta_\ell$.
		Suppose that $E_+(\mathcal{N},K) \ne \varnothing$ and fix $x^* \in E_+(\mathcal{N},K)$. Then
		$$\{x \in \mathbb{R}_{ > 0}^m | {\log \left( x \right) - \log \left( x^* \right)} \in \widetilde{S}^{\perp}\} \subseteq E_+(\mathcal{N},K).$$
		\label{log:subset:equilibria}
	\end{lemma}
	
	\begin{proof}
		Let $x\in \mathbb{R}_{ > 0}^m$ such that $\log \left( x \right) - \log \left( x^* \right) \in \widetilde{S}^{\perp}$. Then we have $$\left\langle {{{\widetilde{Y}}_{y'}} - {{\widetilde{Y}}_{y}},\log \left( x \right) - \log \left( x^* \right)} \right\rangle  = 0$$ for any reaction $(y,y') \in \mathcal{R}$. By Lemma \ref{pi:eq}, ${\pi _y}\left( {x,x^*} \right) = {\pi _{y'}}\left( {x,x^*} \right)$ for any $y \to y' \in \mathcal{R}$. Hence, ${\pi _y}$ does not change within a linkage class to which the complex $y$ belongs, and we denote the value for the $i$th linkage class by $\pi ^i$. Hence, requiring the cycle terminal PL-FSK property in Proposition \ref{cycle:terminal:PLFSK:iso}. Then,
		\begin{align*}
			f\left( x \right) &= \sum\limits_{i = 1}^\ell {{f^i}\left( x \right)}  = \sum\limits_{i = 1}^\ell {\left( {\sum\limits_{\left( {y,y'} \right) \in {\mathcal{R}_i}} {{\kappa _{yy'}}\prod\limits_{s = 1}^m {x_s^{{{\widetilde{Y}}_{sy}}}\left( {y' - y} \right)} } } \right)}\\
			&= \sum\limits_{i = 1}^\ell {\left( {{\pi ^i}\sum\limits_{\left( {y,y'} \right) \in {\mathcal{R}_i}} {{\kappa _{yy'}}\prod\limits_{s = 1}^m {{{\left( {x_s^*} \right)}^{{{\widetilde{Y}}_{sy}}}}\left( {y' - y} \right)} } } \right)}  = \sum\limits_{i = 1}^\ell {\left( {{\pi ^i}{f^i}\left( {{x^*}} \right)} \right)}.
		\end{align*}
	Since $x^* \in E_+(\mathcal{N},K)$, then $f(x^*)=0$. By Proposition \ref{SFRF:linkageclass}, $f^i(x^*)=0$ for each $i=1,\ldots,\ell$. Thus, $f(x)=0$ and $x \in E_+(\mathcal{N},K)$.
	\end{proof}

In the following next two results, we require the assumption of a system to be cycle terminal PL-FSK since we use Lemma \ref{log:subset:equilibria} in the proofs.

	\begin{theorem}
		Let $(\mathcal{N},K)$ be a cycle terminal PL-FSK system that satisfies $\delta=\delta_1+\cdots+\delta_\ell$. Suppose that $E_+ \ne \varnothing$, $\mathcal{P}$ a positive stoichiometric class such that $E_+ \cap \mathcal{P} \ne \varnothing$, and for $x^* \in \mathbb{R}^m_{>0}$ let $$ Q(x^*)=\{x \in \mathbb{R}^m_{>0} | \log(x)-\log(x^*) \in \widetilde{S}^{\perp}\}$$ then
		$(\mathcal{N}, K)$ is a poly-PLP system
		with flux subspace $P_E=\widetilde{S}$ and reference points of the form $x^{*,j}$
		, i.e.,
		\[E_+=\sideset{}{^*}\bigcup\limits_{{x^{*}} \in {E_ + } \cap \mathcal{P}} {Q\left( {{x^*}} \right)}\]
		where $\bigcup ^*$ indicates the usual union but with the condition that if $x^{*,1},x^{*,2} \in E_+ \cap \mathcal{P}$ and $x^{*,1} \ne x^{*,2}$, then $Q(x^{*,1})\cap Q(x^{*,2}) = \varnothing$.
		\label{positive:stoich:class}
	\end{theorem}

	\begin{proof}
	For all $x, x^* \in \mathbb{R}^m_{>0}$, $x \in Q(x^*)$ if and only if $x^{*} \in Q(x)$. By definition of $Q$, we have $x^{*,1}, x^{*,2} \in E_{+} \cap \mathcal{P}$
	and $x \in Q(x^{*,1}) \cap Q(x^{*,2})$ implies $x^{*,1}, x^{*,2}\in Q(x)$. By Lemma \ref{unique:vector:subspace}, $x^{*,1} = x^{*,2}$ implying the disjointness of the sets. By Lemma \ref{log:subset:equilibria}, $$\sideset{}{^*}\bigcup\limits_{{x^*} \in {E_ + } \cap \mathcal{P}} {Q\left( {{x^*}} \right)} \subseteq E_+.$$
	Conversely, suppose $x \in E_+$. By Lemma \ref{unique:vector:subspace}, we are ensured of the existence of unique $x^* \in \mathcal{P} \cap Q(x)$. Then, $x^* \in E_+$ by Lemma \ref{log:subset:equilibria}. Now, $x \in Q(x^*)$, since $x \in Q(x^*)$ if and only if $x^{*} \in Q(x)$. Then, the other containment follows.
	\end{proof}

\begin{example}
	We again consider the reaction network from Jose et al. \cite{JTAM2021} as given in Fig. \ref{fig:Josecounterexample}. The kinetics and $\widehat{T}$-matrix are given as follows:
	\[K\left( x \right) = \left[ {\begin{array}{*{20}{c}}
			{{k_1}{A_2}^{ - 1}{A_3}}\\
			{{k_2}{A_1}^{ - 1}{A_2}^{ - 1}{A_3}}\\
			{{k_2}{A_1}^{ - 1}{A_2}^{ - 1}{A_3}}\\
			{{k_3}{A_2}^{ - 2}}\\
			{{k_4}{A_3}^{ - 2}}
	\end{array}} \right]{\rm{    \ and \      }}\widehat{T} = \left[ {\begin{array}{*{20}{c}}
			0&{ - 1}&0&0\\
			{ - 1}&{ - 1}&{ - 2}&0\\
			1&1&0&{ - 2}\\
			1&1&1&1
	\end{array}} \right].\]
	Since the network has a single linkage class, it has ILC. It is in fact PL-RLK, i.e., the set column vectors of the T-matrix is linearly independent, and thus factor span surjective. According to Theorem \ref{positive:stoich:class}, the system is poly-PLP.
	\label{ex:poly:PLP}
\end{example}

	\begin{proposition}
		Let $(\mathcal{N},K)$ be a cycle terminal PL-FSK that satisfies $\delta=\delta_1+\cdots+\delta_\ell$. Suppose that $E_+ \ne \varnothing$. Then
		\begin{itemize}
			\item[i.] for all positive stoichiometric classes $\mathcal{P}$ and $\mathcal{P}'$, there exists a bijection between $E_+ \cap \mathcal{P}$ and $E_+ \cap \mathcal{P}'$, and
			\item[ii.] $E_+ \cap \mathcal{P}$ is finite for some positive stoichiometric class $\mathcal{P}$ then
				\begin{itemize}
					\item[a.] $E_+ \cap \mathcal{P}'$ is finite for each positive stoichiometric class $\mathcal{P}'$ and
					\item[b.] $|E_+ \cap \mathcal{P}'|=|E_+ \cap \mathcal{P}|$ for each positive stoichiometric class $\mathcal{P}'$.
				\end{itemize}
		\end{itemize}
	\end{proposition}

	\begin{proof}
		The proposition follows from Lemma \ref{unique:vector:subspace} and Theorem \ref{positive:stoich:class}.
	\end{proof}

\section{A criterion for the existence of positive equilibria for $\widehat{T}$-independent PL-RDK systems on networks with ILC}
\label{results2}
In this section, we introduce the concept of $\widehat{T}$-independence of a PL-RDK system and show that, on a network with ILC, the existence of positive equilibria for each linkage class is a necessary and sufficient condition for the property for the whole system. This result extends a result of Boros \cite{boros:paper:thesis,boros:thesis} on mass action systems with ILC and partially a result of Talabis et al. \cite{talabis:positive:eq:jomc} on PL-TIK systems, which form a subset of $\widehat{T}$-independent systems. We first define two important concepts then proceed with the main result of this section.

\begin{definition}
	The {\bf Laplacian matrix} $A_k$ is an $n\times n$ matrix such that 
	$${\left( {{A_k}} \right)_{ij}} = \left\{ \begin{array}{l}
		{k_{ji}}{\rm{  \ \ \ \ \ \ \  \ \ \ \ \ \ \     \ \ \ \                 if \ }}i \ne j\\
		{k_{jj}} - \sum\limits_{l = 1}^n {{k_{jl}}} {\rm{  \ \ \ \ \ \ \       if \ }}i = j
	\end{array} \right.,$$
where $k_{ji}$ is the label (often called rate constant) associated to the reaction from complex $j$ to complex $i$.
\label{def:Laplacian}
\end{definition}

We can see from the first condition in Definition \ref{def:Laplacian} that $k_{ji}=0$ when reaction $i\to j$ is not present in the network.

\begin{definition}
The {\bf factor map} $\psi : \mathbb{R}^m \to \mathbb{R}^n$ is defined as
$(\psi )_y(x) = x^{F_i}$ if
$y$ is a reactant complex of reaction $i$, and 0 otherwise.
\end{definition}


	\begin{theorem}
		Let $(\mathcal{N},K)$ be a PL-RDK system that satisfies $\delta=\delta_1+\cdots+\delta_\ell$ and ${\widehat{T}}={\widehat{T}}^1 \oplus \cdots \oplus {\widehat{T}}^\ell$ (which we call the $\widehat{T}$-independence). Then
		$$E_+ \ne \varnothing {\text{ if and only if }} E_+^i \ne \varnothing$$
		for each $i=1,\ldots,\ell.$
		\label{extension:talabis}
	\end{theorem}

	\begin{proof}
		Suppose $E_+ \ne \varnothing$. By Proposition \ref{SFRF:linkageclass}, $E_+^i \ne \varnothing$ for all $i=1,\ldots \ell$.
		
		Conversely, suppose $E_+^i \ne \varnothing$ for all $i=1,\ldots \ell$. We have $x \in E_+^i \ne \varnothing$ if and only if
		\begin{equation}
		{\text{there exists a }}v^i \in {\mathbb{R}}^{n^i}_{>0}
		{\text{ such that }} \psi_{\rm neg}^i(x)=v^i
		{\text{ and }} Y^i\cdot A_{k,{\rm neg}}^i\cdot v^i=0
		\label{equiv1:nonemp:LC}
		\end{equation}
		where $\psi^i_{\rm neg}$, $Y^i$ and $A_{k,{\rm neg}}^i$ are the factor map (neglecting non-reactant complexes), the molecularity matrix and the Laplacian matrix (without columns for non-reactant complexes) for the $i$th linkage class, respectively.
		
		Now, $\log \left( {{\psi_{\rm neg} ^i}\left( x \right)} \right) = {\left( {{{T}^i}} \right)^\top}\log \left( x \right)$ and $\log:\mathbb{R}^m_{>0} \to \mathbb{R}^m$ is bijective, Condition \eqref{equiv1:nonemp:LC} is equivalent to
		\begin{equation}
			{\text{there exists a }}
			v^i \in {\mathbb{R}}^{n^i}_{>0}
			{\text{ such that }} \log \left( {{v^i}} \right) \in {\mathop{\rm Im}\nolimits} {\left( {{{T}^i}} \right)^\top}
			{\text{ and }} Y^i\cdot A_{k,{\rm neg}}^i\cdot v^i=0.
			\label{equiv2:nonemp:LC}
		\end{equation}
	We denote the vector in $\mathbb{R}^{n^i}$ with all coordinates equal to 1 by ${\bf{1}}^i$. Note that $${\mathop{\rm Im}\nolimits} {\left( {{\widehat{{T}}^i}} \right)^\top} = \left[ {{\mathop{\rm Im}\nolimits} {{\left( {{{T}^i}} \right)}^\top},{{\bf{1}}^i}} \right]$$
	and for each $\gamma^i \in \mathbb{R}_{>0}$, we have
	$\log \left( {{\gamma ^i}{v^i}} \right) = \log \left( {{v^i}} \right) + \log \left( {{\gamma ^i}} \right){{\bf{1}}^i}$,
	and Condition \eqref{equiv2:nonemp:LC} is equivalent to 
	\begin{equation}
		{\text{there exists a }}
		v^i \in {\mathbb{R}}^{n^i}_{>0}
		{\text{ such that }} \log \left( {{v^i}} \right) \in {\mathop{\rm Im}\nolimits} {\left( {{\widehat{{T}}^i}} \right)^\top}
		{\text{ and }} Y^i\cdot A_{k,{\rm neg}}^i\cdot v^i=0.
		\label{equiv3:nonemp:LC}
	\end{equation}
	At this point $x \in E_+^i \ne \varnothing$ for each $i=1,\ldots,\ell$ if and only if Condition \eqref{equiv3:nonemp:LC} holds. Hence,
	\[\left\{ {z \in \mathbb{R}_{ > 0}^{m + 1}|{{\left( {{\widehat{T}}}^i \right)}^\top} \cdot z = \log \left( {{v^r}} \right)} \right\} \ne \varnothing.\]
	By assumption, ${\widehat{T}}={\widehat{T}}^1 \oplus \cdots \oplus {\widehat{T}}^\ell$ and using Lemma \ref{solution:system:ran}, we have
	\[\bigcap\limits_{i=1}^{\ell}{
		\left\{ {z \in \mathbb{R}_{ > 0}^{m + 1}|{{\left( {{\widehat{T}}}^i \right)}^\top} \cdot z = \log \left( {{v^r}} \right)} \right\} \ne \varnothing}\]
	\begin{equation}
	\Longrightarrow \left[ {\begin{array}{*{20}{c}}
			{\log \left( {{v^1}} \right)}\\
			{\vdots}\\
			{\log \left( {{v^l}} \right)}
	\end{array}} \right] = \left[ {\begin{array}{*{20}{c}}
			{{{\left( {{{\widehat{T}}^1}} \right)}^\top}}\\
			{\vdots}\\
			{{{\left( {{{\widehat{T}}^\ell}} \right)}^\top}}
	\end{array}} \right]\left[ {\begin{array}{*{20}{c}}
			u^{\top}\\
			w^{\top}
	\end{array}} \right]
	\label{log:matrix:equation}
	\end{equation}
for some $u \in \mathbb{R}^m$ and $w \in \mathbb{R}^\ell$.
Let $x \in \mathbb{R}^n_{>0}$ and $\gamma^1, \ldots, \gamma^\ell \in \mathbb{R}_{>0}$ such that $\log(x) = u$ and $-\log(\gamma^i) = w_i$ for
each $i = 1,\ldots,\ell$. Equation \eqref{log:matrix:equation} becomes
	\begin{equation*}
	\left[ {\begin{array}{*{20}{c}}
			{\log \left( {{v^1}} \right)}\\
			{\vdots}\\
			{\log \left( {{v^l}} \right)}
	\end{array}} \right] = \left[ {\begin{array}{*{20}{c}}
			{{{\left( {{{\widehat{T}}^1}} \right)}^\top}}\\
			{\vdots}\\
			{{{\left( {{{\widehat{T}}^\ell}} \right)}^\top}}
	\end{array}} \right]
\left[ {\begin{array}{*{20}{c}}
		{\log {{\left( {{x^1}} \right)}^\top}}\\
		{-\log \left( {{\gamma _1}} \right)}\\
		{\vdots}\\
		{-\log \left( {{\gamma _\ell}} \right)}
\end{array}} \right].
	\label{log:matrix:equation2}
\end{equation*}
Hence, for all $i = 1, \ldots \ell$ and for all $y \in \mathcal{C}^i$, we have
\[{\gamma ^i}{v_y^i} = \prod\limits_{s = 1}^m {x_s^{{Y_{sy}}} = \left(\psi _{\rm neg}\right)_y} \left( x \right) \Longrightarrow \gamma^i v^i = {\psi _{\rm neg}}^i\left( x \right).\]
Then,
\begin{align*}
Y \cdot {A_{k,{\rm neg}}} \cdot {\psi _{\rm neg}}\left( x \right) &= \sum\limits_{i = 1}^\ell {{Y^i} \cdot {{\left( {{A_{k,{\rm neg}}}} \right)}^i} \cdot {{\left( {{\psi _{{\rm neg}}}} \right)}^i}\left( x \right)}\\
&= \sum\limits_{i = 1}^\ell {{\gamma ^i} \cdot {Y^i} \cdot {{\left( {{A_{k,{\rm neg}}}} \right)}^i} \cdot {v^i}}\\
& = \sum\limits_{i = 1}^\ell {{\gamma ^i} \cdot 0}\\
& = 0.
\end{align*}
This proves that $x \in E_+$.
\end{proof}

\begin{example}
Let $(\mathcal{N},K)$ be PL-TIK system studied in Talabis et al. (Theorem 4 \cite{talabis:positive:eq:jomc}). The system is a PL-RDK system that satisfies the $\widehat{T}$-independence. If the ILC property holds for the underlying network, i.e., $\delta=\delta_1+\cdots+\delta_\ell$, then the conclusion of Theorem \ref{extension:talabis} is true for the given system.
\end{example}

\section{An application to the Extension Problem for the Horn-Jackson ACB Theorem}
\label{results4}
In this Section, we recall the Extension Problem for the Horn-Jackson ACB Theorem and identify a large class of PL-RDK systems to which the extension does not hold. We further apply our previous results to obtain a simple criterion for weakly reversible PL-FSK systems with ILC to be in the complement  of this class, leading to the identification of a new class of ACB PL-RDK systems.

\subsection{The Extension Problem for the Horn-Jackson ACB Theorem and multi-PLP systems}
A complex balanced kinetic system, i.e., $Z_+(\mathcal{N},K) \ne \varnothing$, is absolutely complex balanced (ACB) if  $E_+(\mathcal{N},K)=Z_+(\mathcal{N},K)$, i.e., every positive equilibrium is complex balanced.
Two ACB Theorems were foundational in the modern CRNT: in 1972, M. Feinberg showed that any complex balanced deficiency zero kinetic system is absolutely complex balanced (Feinberg ACB Theorem). In the same year, F. Horn and R. Jackson derive an extension to positive deficiency systems, showing that any complex balanced mass action system, regardless of deficiency, is absolutely complex balanced (Horn-Jackson ACB Theorem). The extension is partial in the sense that it was limited to mass action kinetics. Fifty years later, Jose et al. \cite{JTAM2021} posed the Extension Problem for the Horn-Jackson ACB Theorem: for which kinetic sets, besides mass action, is every complex balanced system absolutely complex balanced? In view of the Feinberg ACB Theorem, one of course only needs to consider systems with positive deficiency.

A natural candidate for an extension is the set of complex balanced PL-RDK systems for two reasons:
\begin{itemize}
\item[1.] In his 1979 Wisconsin Lecture Notes, M. Feinberg derived the equivalence of the ACB property with the log parametrizability of $Z_+(\mathcal{N},K)$ (denoted as CLP) for mass action systems.
\item[2.] S. M\"uller and G. Regensburger in 2014 in turn showed that any complex balanced PL-RDK system is a CLP system with flux subspace $= \widetilde{S}$.
\end{itemize}
Jose et al. extended the Horn-Jackson result to weakly reversible PL-TIK systems, whose linkage classes are independent (ILC) and have deficiencies 0 or 1. These are precisely the weakly reversible systems which satisfy the conditions of the Deficiency One Theorem for PL-TIK systems \cite{talabis:positive:eq:jomc}. They also constructed the system in Example \ref{ex:poly:PLP} and showed that though CLP, is not a PLP system. From this, they could conclude that it is not absolutely complex balanced, thus providing a counterexample to a general extension of the Horn-Jackson ACB Theorem to PL-RDK systems (or even just PL-TIK systems).

The following Proposition and Corollary identify a large class of PL-RDK systems to which the Horn-Jackson ACB Theorem does not extend:

\begin{proposition}
Let $(\mathcal{N},K)$ be a complex balanced PL-RDK system. If it is poly-PLP with flux subspace $= \widetilde{S}$, then it is non-ACB $\Longleftrightarrow$ it is a multi-PLP system.
\end{proposition}

\begin{proof}
By Proposition 4 of \cite{MURE2014}, it follows that the system is a CLP system with flux subspace $\widetilde{S}$. By Theorem 4 of \cite{JTAM2021}, the system is non-ACB if and only if it is not a PLP system with flux subspace $\widetilde{S}$. Since the system is poly-PLP, the latter is equivalent to the system being multi-PLP.
\end{proof}

\begin{corollary}
Let $(\mathcal{N},K)$ be a weakly reversible multi-PLP system with kinetic deficiency $\widetilde{\delta} = 0$. Then it is complex balanced, but not absolutely complex balanced.
\end{corollary}

\begin{proof}
Since the system is weakly reversible, PL-RDK and has zero kinetic deficiency, it follows from Theorem 1 of \cite{MURE2014} that the system is complex balanced. On other hand, since the system is multi-PLP, from the previous Proposition, the system is not absolutely complex balanced.
\end{proof}

\begin{example}
Since the system from Jose et al. \cite{JTAM2021} in Example \ref{ex:poly:PLP} is a PL-TIK system, it has zero kinetic deficiency [21]. According to Example \ref{ex:poly:PLP}, it is poly-PLP with flux subspace $\widetilde{S}$, but the computations of Jose et al. show that it is not PLP. Hence, it must be multi-PLP and therefore it belongs to the class of PL-RDK systems specified above.
\end{example}

\begin{remark}
Weakly reversible, multi-PLP PL-RDK systems with $\widetilde{\delta}>0$ which satisfy the complex balancing criterion in Theorem 1 (Statement ii) of \cite{MURE2014}, constitute a further class of non-ACB systems.
\end{remark}

\subsection{A criterion for multi-PLP for cycle terminal PL-FSK systems with ILC}

We now apply our previous results to derive a simple criterion and computational procedure to determine when a cycle terminal PL-FSK system with ILC is multi-PLP.

\begin{proposition}
Let $(\mathcal{N},K)$ be a cycle terminal PL-FSK system with ILC and and non-empty $E_+$. Then $(\mathcal{N},K)$ is multi-PLP $\Longleftrightarrow$ $(\mathcal{N},K)$ is multi-stationary.
\end{proposition}

\begin{proof}
For a given stoichiometric class $\mathcal{P}$ and $x^{*,i} \in E_+ \cap \mathcal{P}$,  we have a map $\{x^{*,i}\} \to \{ Q(x^{*,i})\}$. Clearly, the map is surjective. According to Theorem \ref{positive:stoich:class}, in particular the last statement explaining the ``star union'', it is also injective.
Hence, when finite, $|E_+ \cap \mathcal{P}| 
= | \{ Q(x^{*,i})\}|=\mu$. As remarked before, since the $Q(x^{*,i})$ are LP sets with flux subspace $\widetilde{S}$, $|E_+ \cap Q| = \mu$. This establishes the equivalence of multi-PLP and multi-stationarity for cycle terminal PL-FSK systems with ILC.
\end{proof}

\begin{remark}
The Multi-stationarity Algorithm for Power Law Kinetic systems (MSA-PLK) in \cite{HEMD2021} of Hernandez et al. (2019) can hence be used to check for the multi-PLP property. If the union is finite, $\mu = |E_+ \cap \mathcal{P}|$ if the intersection is non-empty. This would have been necessary for Example \ref{ex:poly:PLP} if the previous computations from Jose et al. had not been available.
\end{remark}

\begin{corollary}
Any complex balanced, monostationary PL-FSK system with ILC is absolutely complex balanced.
\end{corollary}

\begin{proof}
Again, by Proposition 4 of \cite{MURE2014} it is a CLP system with flux subspace $\widetilde{S}$ and monostationarity implies that it is PLP with the same flux subspace. By Theorem 4 of \cite{JTAM2021}, it is absolutely complex balanced.
\end{proof}

\section{Absolute concentration robustness in poly-PLP systems}
\label{results3}
The concept of absolute concentration robustness (ACR) in a network's species was introduced by G. Shinar and M. Feinberg in a well-known paper in the journal Science in 2010 \cite{ShinarFeinberg2010}. ACR denotes the invariance of a species' concentration at all positive equilibria of the system.  Shinar and Feinberg also provided a remarkable sufficient condition inspired by experimental observations in subsystems of {\it{E. coli}} for deficiency one mass action systems. This condition was extended to power law kinetic systems of low deficiency (i.e., $\delta$ = 0 or 1) \cite{FLRM2018,FLRM2020}, subsets of poly-PL kinetic systems \cite{LLMM2021}, and Hill-type kinetic systems \cite{HEME2021}. Independent decompositions allowed the extension of the Shinar-Feinberg criterion to larger systems and those with higher deficiency.

In this Section, we first formulate a general necessary and sufficient condition, called an ACR species hyperplane criterion, for ACR in a species for any kinetic system.  We then show that when applied to the systems we have studied in the previous sections, it leads to a relatively simple procedure for determining ACR in the species of the systems.

\subsection{A general species hyperplane criterion for ACR in kinetic systems}
Mathematically, the property of ACR in a species $S$ is equivalent to the statement that all positive equilibria of a kinetic system lie in a hyperplane $x_S = c$, with a positive constant $c$ (this is the so-called ``ACR hyperplane''). Clearly, this is equivalent to the statement that for any two equilibria $x^*$ and $x^{**}$, $x^* - x^{**}$ lies in the hyperplane $x_S = 0$.


Let $(\mathcal{N}, K)$ be a kinetic system, with $\mathcal{N} = (\mathcal{S, C, R})$ and $m$ is the number of species.

\begin{definition}
For any species $S$, the $(m-1)$-dimensional subspace
$$H_S := \{ x \in \mathbb{R}^{\mathcal{S}}|x_S = 0 \}$$
is called the species hyperplane of $S$.
\end{definition}

For $U$ containing $\mathbb{R}_{>0}$ , let $\phi: U \to \mathbb{R}$ be an injective map, i.e., $\phi: U \to {\rm{Im \ }} \phi$ is a bijection.
By component-wise application ($m$ times), we obtain a bijection $U^\mathcal{S} \to \mathbb{R}^\mathcal{S}$, which we also denote with $\phi$.
We formulate our concepts for any subset $Y$ of $E_+(\mathcal{N}, K)$, although we are mainly interested in $Y =  E_+(\mathcal{N}, K)$.

\begin{definition}
For a subset $Y$ of $E_+(\mathcal{N}, K)$, the set  $$\Delta_{\phi}Y := \{ \phi(x) - \phi(x') | x, x' \in Y\}$$ is called the difference set of $\phi$-transformed equilibria in $Y$, and its span
$\left\langle {\Delta_{\phi}Y} \right\rangle$
the difference space of $\phi$-transformed equilibria in $Y$.
\end{definition}

Recall that a system has concentration robustness in a species $S$ over a subset $Y$ of positive equilibria if the concentration value for $S$ is invariant over all equilibria in $Y$. Most studied has been absolute concentration robustness where $Y = E_+(\mathcal{N}, K)$, and balanced concentration robustness when $Y = Z_+(\mathcal{N}, K)$.

The following Proposition formulates the general species hyperplane criterion for concentration robustness for a subset of positive equilibria $Y$:
 
\begin{proposition} Let $(\mathcal{N}, K)$ be a kinetic system and $U$ a superset of $\mathbb{R}_{>0}$.
\begin{enumerate}
\item[i.] $(\mathcal{N}, K)$ has concentration robustness in a species $S$ over $Y$ if and only if there is an injective map an injective map $U \to \mathbb{R}$ such that  $\left\langle {\Delta_{\phi}Y} \right\rangle$ is contained in $H_S$.
\item[ii.] If $m_{\text{ACR}}$ $(m_{\text{BCR}})$
is the number of species with ACR (BCR),
then 
$$m_{\text{ACR}} \le m - \dim \left\langle {\Delta_{\phi}E_+} \right\rangle (m_{\text{BCR}} \le m - \dim \left\langle {\Delta_{\phi}Z_+} \right\rangle).$$
\end{enumerate}
\end{proposition}

\begin{proof}
Statement (i) follows directly from the definitions, i.e., concentration robustness in $S$ over $Y$ $\iff$ for all $x, x' \in Y$, $x_S = x'_S \iff \phi(x_S) = \phi(x'_S)$ for any bijection $\phi$
(e.g., $\phi = \log$) $\iff$ $\phi (x_S) - \phi (x'_S) = 0 \iff \left\langle {\Delta_{\phi}Y} \right\rangle$ is contained in $H_S$.
The claims in (ii) follow from (i), since $\left\langle {\Delta_{\phi}E_+} \right\rangle$ is contained in $\bigcap H_S$. Hence,
$\dim \left\langle {\Delta_{\phi}E_+} \right\rangle \le m - m_{\text{ACR}}$, and analogously for $Z_+$ and BCR.
\end{proof}

\begin{corollary}
If $(\mathcal{N}, K)$ has ACR in at least one species, then $\left\langle {\Delta_{\phi}E_+} \right\rangle$ does not contain any positive vector.
\label{cor:posvector:ACR}
\end{corollary}

\begin{remark}
Note that any injective map $U^\mathcal{S} \to \mathbb{R}^\mathcal{S}$ leads  to an ACR hyperplane criterion. 
The key question for the choice of the map $\phi$ is the tractability of computing $\left\langle {\Delta_{\phi}E_+} \right\rangle$.
\end{remark}

\begin{example}
For monostationary systems \cite{FAME2022}, the choice for $U = \mathbb{R}$ and $\phi = {\rm{id}}$, so that ${\rm{Im \ }} \phi = \mathbb{R}$.
\end{example}


\subsection{Concentration robustness in poly-PLP systems}
We now apply the general species hyperplane criterion to the logarithm function.

\begin{definition}
The subspace ${P_E}^{*} = \left\langle \log(x_i^{*}) - \log(x_j^{*}) \right\rangle ^{\perp}$ is called the {\bf{reference equilibria flux subspace}}.
\end{definition}

The ACR species hyperplane criterion for poly-PLP systems is the following:

\begin{theorem}
Let $(\mathcal{N}, K)$ be a poly-PLP system with flux space $P_E$, reference equilibria $\{{x^{*}_i}\}$ and reference equilibria flux subspace ${P_E}^{*}$.  Then 
\begin{enumerate}
	\item[i.] it has ACR in a species $S$ if and only if 
	${P_E}^{\perp} + ({P_E}^{*})^{\perp}=({P_E} \cap {P_E}^{*})^{\perp}$ is contained in the hyperplane $\{x |x_S = 0\}$, and
	\item[ii.] if $m_{\text{ACR}}$ is the number of species with ACR property, then $$m_{\text{ACR}} \le \dim ({P_E} \cap {P_E}^{*}) \le \min\{\dim({P_E}),\dim({P_E}^{*})\}.$$
\end{enumerate}
\label{ACR:hyperplane:criterion:multiPLP}
\end{theorem}

\begin{proof}
\begin{itemize}
	\item[i.] We only need to show that $\left\langle \Delta_{\log}E_+\right\rangle={P_E}^{\perp} + ({P_E}^{*})^{\perp}$.
	We denote the set $\{x \in \mathbb{R}^{\mathcal{S}}_{>0} | \log (x) - \log (x_i^{*}) \in {P_E}^{\perp}\}$ with $E_i$.
	If $x \in E_i$ and $x' \in E_j$, then $\log (x) - \log (x') = (\log (x_i^{*}) - \log (x_j^{*})) + (p_i - p_j) \in ({P_E}^{*})^{\perp} + {P_E}^{\perp}$.
	Hence, $\left\langle \Delta_{\log}E_+\right\rangle \subseteq {P_E}^{\perp} + ({P_E}^{*})^{\perp}$.
	Conversely, if $p \in {P_E}^\perp$, for each pair $(i, j)$, since the system is poly-PLP, there exists an $x_i \in E_i$ such that $\log (x_i) = \log(x_i^{*}) + \dfrac{p}{2}$, and analogously an $x_j \in E_j$ such $\log (x_j) = \log(x_j^{*}) - \dfrac{p}{2}$.
	Hence, $(\log(x_i^*) - \log(x_i^*)) + p = \log (x_i) - \log (x_j)$, implying
	${P_E}^{\perp} + ({P_E}^{*})^{\perp} \subseteq \left\langle \Delta_{\log}E_+\right\rangle$.
	Applying the general species hyperplane criterion delivers the claims. 
	 
	\item[ii.] We have $({P_E} \cap {P_E}^{*})^{\perp}$ is contained in $\bigcap\limits_S {\left\{ x|{{x_S} = 0} \right\}}$ for all ACR species, so that $m - \dim ({P_E} \cap {P_E}^{*}) \le m - m_{\text{ACR}}$, which leads to the claim.
\end{itemize}
\end{proof}

Theorem \ref{ACR:hyperplane:criterion:multiPLP} suggests that a way to determine 
ACR in a poly-PLP system is to compute bases of 
${P_E}^{\perp}$ and $({P_E}^{*})^{\perp}$,
respectively, and then observe which species coordinates are zero in all basis vectors of the two spaces. These are precisely the ACR species.

\begin{remark}
Note that if $\mu = |E_+ \cap Q|$ is finite, then, for any species $S$,  $$({P_E}^{*})^\perp \subseteq H_S \iff \left\{ \log \dfrac{x^{*}_j}{x^{*}_1} | j=2,\ldots,\mu\right\} \subseteq H_S.$$
\end{remark}

\begin{example}
If $(\mathcal{N}, K)$ is a (mono-)PLP system, then $({P_E}^{*})^\perp = 0$, and we recover the results of \cite{LLMM2021}, including a simple computational procedure for the species hyperplane criterion. In that paper, the bijection $L_{x^{*}}(x) := \log (x) - \log (x^{*})$ was used to derive the results.
\end{example}

\begin{example}
We now refer to Fig. \ref{fig:HornJacksonExample}.
The mass action system was first studied by F. Horn and R. Jackson \cite{HORN1972:GMAK}. Note that $\varepsilon >0$ is a parameter that depicts rate constant.
\begin{figure}
	\centering
	\includegraphics[width=7cm,height=4cm,keepaspectratio]{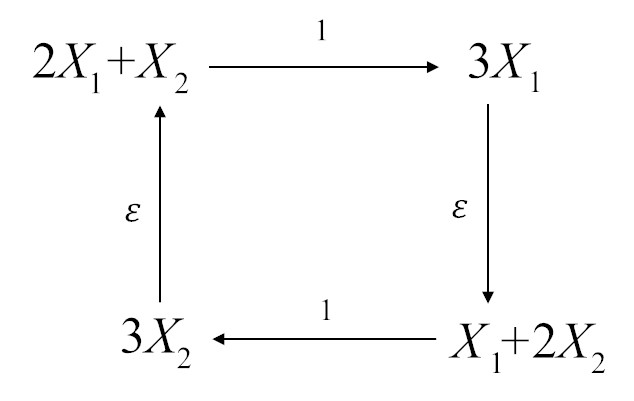}
	\caption{A mass action system first studied by F. Horn and R. Jackson from \cite{HORN1972:GMAK}}
	\label{fig:HornJacksonExample}
\end{figure}

The stoichiometric subspace $S = \left\langle {{{\left( {1, - 1} \right)}^\top}} \right\rangle $, and hence,
$\delta = 4 - 1 - 1 = 2$. Since $\ell = 1$, the linkage class decomposition is trivially independent, so it is poly-PLP. In \cite{boros:thesis}, B. Boros computed the positive equilibria sets and showed that the system is tri-PLP if $0 < \varepsilon < \dfrac{1}{6}$ and mono-PLP if $\dfrac{1}{6} < \varepsilon$. However, for the ACR analysis, we do not need the equilibria computations because the network is conservative (the vector $(1,1)^\top$ is clearly in $S^\perp$), so that Corollary \ref{cor:posvector:ACR} implies that it does not have ACR in any species.
\end{example}
	
\section{Summary and Outlook}
\label{summary}
	The summary of our work and an outlook for future direction are given as follows:
\begin{enumerate}
	\item We establish that any PL-FSK system with non-empty set of equilibria such that the underlying network $\mathcal{N}$ that is cycle terminal (each complex is a reactant complex) and has independent linkage classes (ILC), is a poly-PLP  system, i.e., its set of positive equilibria is the disjoint union of log-parametrized sets with kinetic order subspace, introduced by M\"uller and Regensburger, as the flux subspace. This result extends the results of Boros in his PhD thesis on mass action systems with ILC.
	The key insight for the extension is that factor span surjectivity induces an isomorphic digraph structure on the kinetic complexes, and hence a bijection between their decompositions. This result also completes, for ILC networks, the structural analysis of the original complex balanced GMAS by M\"uller and Regensburger.
	\item We identify a large set of PL-RDK systems for which the non-emptiness of its set of equilibria is a necessary and sufficient condition for the non-emptiness of each set of positive equilibria of the subsystems with underlying networks as the linkage classes, and hence extending the result of Boros on mass action systems with ILC.
	\item We show that for complex balanced PL-RDK systems which are poly-PLP with the kinetic order subspace as the flux subspace, the multi-PLP systems are precisely those which are not absolutely complex balanced (ACB). Additionally, we identify the set of monostationary PL-FSK systems as a new class of PL-RDK systems to which the Horn-Jackson result on ACB extends. Moreover, we apply our results for cycle terminal PL-FSK systems to show that for these systems, multi-PLP is equivalent to multi-stationarity.
	\item We use our results to study absolute concentration robustness (ACR) by obtaining a species hyperplane containment criterion to determine ACR in the system species.
	\item One can deal with extensions of results for poly-PL kinetics (non-negative linear combination of power law kinetics) and Hill-type.
\end{enumerate}

\section*{Acknowledgement}
This work was supported by the Institute for Basic Science in Daejeon, Republic of Korea with funding information IBS-R029-C3.



\appendix
\section{List of abbreviations}
\begin{tabular}{ll}
\noalign{\smallskip}\hline\noalign{\smallskip}
Abbreviation& Meaning \\
\noalign{\smallskip}\hline\noalign{\smallskip}
CLP& complex balanced equilibria log parametrized\\
CRN& chemical reaction network\\
GMAS& generalized mass action system\\
ILC& independent linkage classes\\
PL-FSK& power law with factor span surjective kinetics\\
PLK& power law kinetics\\
PL-NDK& power law with non-reactant-determined kinetics\\
PL-NIK& power law with non-inhibitory kinetics\\
PLP& positive equilibria log parametrized\\
PL-RDK& power law with reactant-determined kinetics\\
PL-RLK& power law with reactant set linear independent kinetics\\
PL-TIK& power law with $\widehat{T}$-rank maximal kinetics\\
\noalign{\smallskip}\hline
\end{tabular}

\end{document}